\documentclass{article}
\usepackage{amsmath,amsthm,amssymb,amscd}

\newcommand{\gw}{\omega}

\newcommand{\gs}{\sigma}

\newcommand{\coll}{\mathrm{Coll}}

\newcommand{\cantor}{2^\gw}

\newcommand{\rng}{\mathrm{rng}}
\newcommand{\power}{\mathcal{P}}

\newtheorem{theorem}{Theorem}[section]

\newtheorem{claim}[theorem]{Claim}
\newtheorem{corollary}[theorem]{Corollary}

\theoremstyle{definition}
\newtheorem{definition}[theorem]{Definition}

\title{Sequential topologies and Dedekind finite sets\footnote{2010 AMS subject classification 03E15, 03E25, 03E35.}}

\author{
Jind{\v r}ich Zapletal\\
University of Florida}

\begin{document}
\maketitle

\begin{abstract}
It is consistent with ZF set theory that the Euclidean topology on $\mathbb{R}$ is not sequential, yet every infinite set of reals contains a countably infinite subset. This answers a question of Gutierres.
\end{abstract}

\section{Introduction}

This note deals with sequentiality of the Euclidean topology of the reals in choiceless context. To remind the reader of the standard definitions:

\begin{definition}
Let $X$ be a topological space. 

\begin{enumerate}
\item A set $C\subset X$ is \emph{sequentially closed} if every countable converging sequence of elements of $C$ converges to an element of $C$.
\item The space $X$ is \emph{sequential} if every sequentially closed subset of $X$ is closed.
\end{enumerate}
\end{definition}

\noindent Under small fragments of the axiom of choice, such as the axiom of dependent choices (DC), it is easy to check the status of sequentiality of basic topological spaces. For example, all metric spaces are sequential, while $\gw_1+1$ with the order topology is not. Without the axiom of dependent choices, sequentiality becomes an issue even in the most basic contexts. Gutierres \cite{gutierres:sequential} showed that in ZF, sequentiality of the Euclidean topology of $\mathbb{R}$ has many equivalent restatements, and it implies that every infinite subset of $\mathbb{R}$ contains a countably infinite subset. He asked whether the opposite implication is provable in ZF. In this note, I show that this is not the case.

\begin{theorem}
\label{1theorem}
It is consistent relative to an inaccessible cardinal that ZF holds,  every infinite subset of $\mathbb{R}$ contains a countably infinite subset, and the Euclidean topology of $\mathbb{R}$ is not sequential.
\end{theorem}

\noindent In the interests of brevity, the theorem understates the understanding of its associated model somewhat. In particular, in it every infinite set contains a countably infinite subset, and no uncountable Polish space is sequential. The inaccessible cardinal assumption is used only to make the construction fit under the umbrella of geometric set theory \cite{z:geometric}; I do not know if it is necessary. It is also unclear if it is possible to distinguish between sequentiality of various uncountable Polish spaces; one rather egregious example is $\mathbb{R}$ and $\mathbb{R}^2$.

The paper uses standard set theoretic notation as in \cite{jech:newset}; in matters of geometric set theory it follows \cite{z:geometric}.

\section{Proof of the main theorem}

The model for Theorem~\ref{1theorem} is a forcing extension of the classical choiceless Solovay model by a certain Suslin poset which is designed in a straightforward way to add a witness for the failure of sequentiality of $\mathbb{R}$.

\begin{definition}
The poset $P$ consists of all pairs $p=\langle a_p, b_p\rangle$ where $a_p\subset [0, 1]$ is a nowhere dense closed set and $b_p\subset [0, 1]$ is a countable set disjoint from $a_p$. The ordering is defined by $q\leq p$ if $a_p\subset a_q$ and $b_p\subset b_q$. The $P$-name $\dot A$ is defined as the union of all first coordinates of conditions in the generic filter.
\end{definition}

\noindent I make a couple of simple initial observations. First, conditions $p, q\in P$ are compatible if and only if $a_p\cap b_q=0$ and $a_q\cap b_p=0$. In such a case, there is a largest common lower bound of $p, q$, namely the condition $\langle a_p\cup a_q, b_p\cup b_q\rangle$. It follows immediately that the poset $P$ is Suslin. Since any point not in $b_p$ can be added to $a_p$ and any point not in $a_p$ can be added to $b_p$ obtaining a condition stronger than $p$, it also follows that $P$ forces the set $\dot A\subset [0, 1]$ to be dense with dense complement; in particular, $\dot A$ is forced not to be closed.

The partial order $P$ is fairly useless in ZFC context. However, I will show that if $W$ is a choiceless Solovay model, then the $P$-extension of $W$ is a model for the theory required by Theorem~\ref{1theorem}. In order to do that, an analysis of its balanced virtual conditions as in \cite{z:geometric} is necessary. This analysis takes place in ZFC. For every nowhere dense closed set $a\subset [0, 1]$, let $\tau_a$ be the $\coll(\gw, [0, 1])$-name for the condition $\langle a, [0,1]\cap V\setminus a\rangle\in P$.

\begin{theorem}
\label{balancetheorem}
In the poset $P$:

\begin{enumerate}
\item for every nowhere dense closed set $a\subset [0, 1]$, the pair $\langle \coll(\gw, [0,1]), \tau_a\rangle$ is balanced;
\item for every balanced pair $\langle Q, \gs\rangle$ there is a nowhere dense closed set $a\subset [0,1]$ such that the balanced pairs $\langle Q, \gs\rangle$ and $\langle \coll(\gw, [0,1]), \tau_a\rangle$ are equivalent;
\item distinct closed sets yield inequivalent balanced pairs.
\end{enumerate}

\noindent In particular, the poset $P$ is balanced.
\end{theorem}

\begin{proof}
For (1), suppose that $R_0, R_1$ are posets and $\gs_0, \gs_1$ are their respective names for conditions in $P$ which are stronger than $\langle a, [0,1]\cap V\setminus a\rangle$; I must show that $R_0\times R_1\Vdash\gs_0, \gs_1$ are compatible in $P$. Let $\gs_0=\langle\dot a_0, \dot b_0\rangle$ and $\gs_1=\langle\dot a_1, \dot b_1\rangle$.

\begin{claim}
$R_0\times R_1\Vdash\dot a_0\cap\dot b_1=0$ and $\dot a_1\cap\dot b_0=0$.
\end{claim}

\begin{proof}
I will prove the latter conjunct; the proof of the former is symmetric. Suppose towards a contradiction that $\theta$ is an $R_0$-name for an element of $\dot b_0$ and $r_0\in R_0$, $r_1\in R_1$ are conditions which force in the product that $\theta\in\dot a_1$ holds. Let $M$ be a countable elementary submodel of a large structure containing $r_0, \theta$ in particular. Let $g\subset R_0$ be a filter generic over the model $M$ containing the condition $r_0$ and let $x=\theta/g$. Note that $R_0\Vdash\dot b_0\cap a=0$; so, it must be the case that $x\notin a$. Now, $R_1\Vdash \check x\notin \dot a_1$; so, there has to be a condition $r'_1\leq r_1$ and a basic open neighborhood $O\subset [0, 1]$ of $x$ such that $r'_1\Vdash\dot a_1\cap O=0$. By the genericity of the filter $g$, there has to be a condition $r'_0\leq r_0$ in $g$ such that $r'_0\Vdash\theta\in O$. Then, the condition $\langle r'_0, r'_1\rangle$ forces in the product that $\theta\notin \dot a_1$ holds, contradicting the initial assumptions.
\end{proof}

\noindent It follows immediately from the claim that the product forces $\langle \dot a_0\cup \dot a_1, \dot b_0\cup\dot b_1$ to be a common lower bound of $\gs_0, \gs_1$. Item (1) follows.

For (2), let $\theta$ be the $Q$-name for the first coordinate of $\gs$, and let $a=\{x\in [0, 1]\colon Q\Vdash\check x\in \theta$. I will show that $a\subset [0, 1]$ is closed and the balanced pairs $\langle Q, \gs\rangle$ and $\langle \coll(\gw, [0,1]), \tau_a\rangle$ are equivalent. It is immediate that $a\subset [0, 1]$ is closed and nowhere dense, since $Q\Vdash\theta$ is closed and nowhere dense. To conclude the proof, by \cite[Proposition 5.2.6]{z:geometric} it is enough to show that $Q\Vdash a\subset\theta$ and $([0, 1]\cap V\setminus a)\cap \theta=0$ because then $\gs$ and $\langle a, [0, 1]\cap V\setminus a\rangle$ are forced to be compatible conditions in $P$.

To show that $Q\Vdash a\subset\theta$ holds, suppose that some condition $q$ forces the contrary. Then there has to be a basic open set $O\subset [0, 1]$ such that $a\cap O\neq 0$ and $q\Vdash\theta\cap O=0$. Pick a point $x\in a\cap O\cap V$ and use the definition of the set $a$ to conclude that $Q\Vdash\check x\in\theta$. This immediately contradicts the assumption that  $q\Vdash\theta\cap O=0$.

To show that $Q\Vdash ([0, 1]\cap V\setminus a)\cap\theta=0$ holds, suppose towards a contradiction that there is a point $x\in V\setminus a$ and a condition $q_0\in Q$ forcing $\check x\in\theta$. By the definition of the set $a$, there has to be a condition $q_1\in Q$ forcing $\check x\notin\theta$. Let $H_0, H_1\subset Q$ be mutually generic filters containing the conditions $q_0, q_1$ respectively, and let $p_0=\gs/H_0$ and let $p_1$ be the condition obtained from $\gs/H_1$ by adding the point $x$ to its second coordinate. It is clear that $p_0, p_1$ are conditions incompatible in $P$ as witnessed by the point $x$. This contradicts the initial balance assumption on the pair $\langle Q, \gs\rangle$.

(3) is immediate. To get the last sentence, note that for every condition $p=\langle a_p, b_p\rangle$, the pair $\langle \coll(\gw, [0,1]), \tau_a\rangle$ represents a balanced virtual condition stronger than $p$.
\end{proof}

\begin{corollary}
Let $W$ be the choiceless Solovay model. Then $W\models P\Vdash \dot A\subset\mathbb{R}$ is sequentially closed.
\end{corollary}

\begin{proof}
Suppose towards a contradiction that this fails. Recall that by \cite[Theorem 9.1.1]{z:geometric} and the fact that the poset $P$ is balanced in ZFC, the $P$-extension of the model $W$ contains no new reals and no new $\gw$-sequences of reals. Thus, it must be the case that there is a condition $p\in P$ and an $\gw$-sequence $x$ for elements of $[0, 1]$ converging to some $y$ such that $p\Vdash \rng(\check x)\subset\dot A$ and $\check y\notin \dot A$. The contradiction is achieved in two complementary cases. First, assume that $\rng(x)\subset a_p$. Then, by the closure of $a_p$, it is also true that $y\in a_p$ and $p\Vdash\check y\in\dot A$, contradicting the original assumptions. Second, assume that $\rng(x)\setminus a_p$ is nonempty, containing some point $z$. Consider the condition $q\leq p$ obtained from $p$ by adding $z$ to its second coordinate. Then $q\Vdash\rng(\check x)\not\subset\dot A$, again contradicting the initial assumptions.
\end{proof}

\noindent As pointed out above, $\dot A$ is forced to be dense codense in $[0, 1]$, so not closed. Thus, in the $P$-extension of the choiceless Solovay model, $\mathbb{R}$ is not a sequential space. All that remains to be proved is that in that extension, every infinite set of reals contains a countably infinite subset. This is in fact the main contribution of this paper. Theorem~\ref{1theorem} is an immediate corollary.

\begin{theorem}
Let $W$ be the choiceless Solovay model. Then $W\models P\Vdash$ every infinite set of reals has a countably infinite subset.
\end{theorem}

\begin{proof}
Let $\kappa$ be an inaccessible cardinal, and let $W$ be the choiceless Solovay model derived from $\kappa$. Work in the model $W$. Let $\tau$ be a $P$-name for an infinite set of reals, and let $p\in P$ be a condition. I have to produce an injective function $\pi\colon \gw\to\mathbb{R}$ and a strengthening of the condition $p$ which forces $\rng(\pi)\subset\tau$. To this end, pick a parameter $z\in\cantor$ such that $\tau, p$ are both definable from $z$ and some additional parameters in the ground model. Let $V[K]$ be an intermediate generic extension obtained by a poset of cardinality smaller than $\kappa$ such that $z\in V[K]$. Work in $V[K]$.

Let $\bar p$ be the balanced virtual condition in $P$ associated with the set $a_p$ as in Theorem~\ref{balancetheorem}. If in the model $W$, $\bar p\Vdash\tau\subset V[K]$ holds, then $\bar p\Vdash\tau$ is countable, and the proof is complete. If this fails, then there must be a poset $Q$ of cardinality smaller than $\kappa$, a $Q$-name $\gs$ for a condition stronger than $\bar p$ and $\theta$ and a $Q$-name $\theta$ for a real not in $V[K]$ such that $Q\Vdash\coll(\gw,  <\kappa)\Vdash\gs\Vdash_P\theta\in\tau$. Move back to $W$.

First, let $H_\gw\subset Q$ be a filter generic over $V[K]$. Then, let $\langle H_n\colon n\in\gw\rangle$ be a sequence of filters on $Q$ which are pairwise mutually generic over $V[K][H_\gw]$ and such that $\lim_nH_n=H_\gw$ in the usual topology of $\power(Q)$. This means that for every condition $q\in H_\gw$, for all but finitely many $n\in\gw$ $q\in H_n$ holds. Write $p_n=\gs/H_n$, $p_n=\langle a_n, b_n\rangle$, and $x_n=\theta/H_n$ for every index $n\in\gw$. It will be enough to show that the points $x_n$ for $n\in\gw$ are pairwise distinct and the conditions $p_n$ for $n\in\gw$ have a common lower bound in $P$, because then in the model $W$, that lower bound forces $\{x_n\colon n\in\gw\}\subset\tau$ as desired.

First of all, it is clear that the points $x_n$ for $n\in\gw$  are pairwise distinct, since if $n\neq m$ then $x_n\in V[K][H_n]\setminus V[K]$ and $x_m\in V[K][H_m]\setminus V[K]$, and the models $V[K][H_n], V[K][H_m]$ are mutually generic over $V[K]$. To see why the conditions $p_n$ for $n\in\gw$ have a common lower bound is more difficult. Write $p_\gw=\gs/H_\gw$, $p_\gw=\langle a_\gw, b_\gw\rangle$.

\begin{claim}
$c=a_\gw\cup\bigcup_{n\in\gw}a_n$ is a closed nowhere dense subset of $[0, 1]$.
\end{claim}

\begin{proof}
For the closure I use the fact that $\lim_nH_n=H_\gw$. Suppose that $y\notin c$ is a point. Since $a_\gw$ is a closed set, there must be a basic open neighborhood $O_\gw$ of $y$ which is disjoint from $a_\gw$. There must be a condition $q\in H_\gw$ such that $q\Vdash O_\gw$ is disjoint from the first coordinate of the condition $\gs\in P$. Find a number $n\in\gw$ such that for all $m>n$, $q\in H_m$ holds. Then for all $m>n$, $a_m\cap O_\gw=0$. For each index $m\leq n$, the set $a_m$ is closed and does not contain the point $y$, so there is a basic open neighborhood $O_m$ of $y$ disjoint from $a_m$. Clearly, $\bigcap_{m\leq n}O_m\cap O_\gw$ is an open neighborhood of $y$ disjoint from the set $c$.

To see that $c\subset [0,1]$ is nowhere dense, note that the set $[0,1]\cap V[K]\setminus a$ is dense and the set $c$ contains none of its points. Alternately, $c$ is a union of countably many closed nowhere dense sets, therefore meager and (as a closed set) nowhere dense.
\end{proof}

\noindent Now, by the balance of the virtual condition $\bar p$, for all distinct indices $n, m\leq\gw$ it is the case that $p_n, p_m$ are compatible in $P$. This is to say that $a_n\cap b_m=0$ and $a_m\cap b_n=0$. It follows that $\bigcup_{n\leq\gw}a_n$ and $\bigcup_{n\in\gw}b_n$ are disjoint sets; they together form the requested common lower bound of all conditions $p_n$ for $n\in\gw$ or even $n\leq\gw$.
\end{proof}

\bibliographystyle{plain}
\bibliography{odkazy,zapletal}

\begin{thebibliography}{1}

\bibitem{gutierres:sequential}
Gon\c{c}alo Gutierres.
\newblock Sequential topological conditions in $\mathbb{R}$ in the absence of
  the axiom of choice.
\newblock {\em Math. Log. Quart.}, 49:293--298, 2003.

\bibitem{jech:newset}
Thomas Jech.
\newblock {\em Set Theory}.
\newblock Springer Verlag, New York, 2002.

\bibitem{z:geometric}
Paul Larson and Jind{\v r}ich Zapletal.
\newblock {\em Geometric set theory}.
\newblock AMS Surveys and Monographs. American Mathematical Society,
  Providence, 2020.

\end{thebibliography}

\end{document}